\documentclass[12pt,a4paper]{article}
\usepackage{amsmath,amssymb,amsthm,url,eucal}
\pdfoutput=1
\usepackage[
  pdftitle={Bent Rectangles},
  pdfauthor={Sergey Agievich},
  bookmarks=true,pdfpagemode=None,pdfstartview=FitH]{hyperref}

\renewcommand{\vec}[1]{{\bf #1}}
\newcommand{\tinybox}{\vrule\vbox to.333em{\hrule width.25em\vfil\hrule}\vrule}
\newcommand{\up}[2]{\mathop{\mathstrut #1}\limits^{#2}\mathstrut}
\newcommand{\circup}[1]{\up{#1}{\circ}}
\newcommand{\hatup}[1]{\up{#1}{\mbox{\tiny$\wedge$}}}
\newcommand{\boxup}[1]{\up{#1}{\mbox{\tinybox}}}

\newcommand{\GL}{\mathop{{\bf GL}}\nolimits}
\newcommand{\AGL}{\mathop{{\bf AGL}}\nolimits}
\newcommand{\func}[1]{{\mathcal #1}}
\newcommand{\ZZ}{{\mathbb Z}}
\newcommand{\FF}{{\mathbb F}}
\newcommand{\CC}{{\mathbb C}}
\newcommand{\GaussCoeff}[3]{\left(\genfrac{}{}{0pt}{}
 {\mathstrut#1}{\mathstrut#2}\right)_{#3}}

\theoremstyle{plain}
\newtheorem{lemma}{Lemma}
\newtheorem{proposition}{Proposition}

\theoremstyle{definition}
\newtheorem{example}{Example}

\begin{document}
\begin{sloppypar}

\title{Bent Rectangles\footnote{
Submitted to:
{\it Proceedings of the NATO Advanced Study Institute on
Boolean Functions in Cryptology and Information Security
(Moscow, September 8--18, 2007)}.
}}

\author{Sergey Agievich\\
\small National Research Center
for Applied Problems of Mathematics and Informatics\\[-0.8ex]
\small Belarusian State University\\[-0.8ex]
\small Nezavisimosti av. 4, 220030 Minsk, Belarus\\[-0.8ex]
\small \texttt{agievich@bsu.by}
}

\date{}

\maketitle

\begin{abstract}
We study generalized regular bent functions using a representation
by bent rectangles, that is, special matrices with restrictions 
on rows and columns.
We describe affine transformations of bent rectangles,
propose new biaffine and bilinear constructions,
study partitions of a vector space into affine planes of the same dimension 
and use such partitions to build bent rectangles.
We illustrate the concept of bent rectangles by examples for the Boolean case. 
\end{abstract}

%\begin{keyword}
%Walsh--Hadamard transform, bent function, regular bent function,
%bent rectangle, partition into affine planes.
%\end{keyword}

%%%%%%%%%%%%%%%%%%%% Introduction %%%%%%%%%%%%%%%%%%%%%%%%%%%%%%%%%%%

\section*{Introduction}\label{BRECTS.Intro}

Boolean bent functions were introduced by Rothaus~\cite{BRECTS.Rot76}
and Dillon~\cite{BRECTS.Dil72}, 
and since then have been widely studied owing to their 
interesting algebraic and combinatorial properties 
and because of their applications in signal processing, coding theory, 
cryptography. 
The current results on the construction, classification, enumeration
of bent functions can be found in the 
surveys~\cite{BRECTS.Car00,BRECTS.DobLea04}.

In the present paper we study bent functions using a representation 
by bent rectangles. 
Rows and normalized columns of such rectangles consist 
of the Walsh--Hadamard coefficients of Boolean functions
and the Boolean functions of rows are simply restrictions of a target 
bent function to a subset of variables.

The notion of bent rectangles was introduced in~\cite{BRECTS.Agi02}
and has analogues in other papers.
For example, a special kind of rectangles, 
all elements of which are equal in magnitude,
was proposed in~\cite{BRECTS.AdaTav90}.
Authors of~\cite{BRECTS.CanCha03} also intensively dealt with restrictions 
of~bent functions and actually used 2- and 4-row bent rectangles.
In~\cite[preliminary version]{BRECTS.DobLea05}, 
$\ZZ$-bent squares that are defined for integer-valued functions
are used to construct Boolean bent functions recursively. 

We extend the results of~\cite{BRECTS.Agi02} in the following directions. 
In Sections~\ref{BRECTS.Funcs}, \ref{BRECTS.Rects} 
we transfer the concept of bent rectangles
to the generalized regular bent functions 
over an arbitrary quotient ring of integers (see \cite{BRECTS.Kum85}). 
In Section~\ref{BRECTS.ATrans} we describe affine transformations 
of bent rectangles. 
In Section~\ref{BRECTS.Bi} we propose new constructions of bent functions
which are based on biaffine and bilinear mappings.
In Section~\ref{BRECTS.APart} 
we study partitions of a vector space into affine planes of 
the same dimension and use such partitions to construct bent rectangles.
Additionally, in Section~\ref{BRECTS.Examples} 
we illustrate some known properties of Boolean bent functions 
by using bent rectangles.

The author plans to continue this paper and discuss further
dual bent functions, bent rectangles with a small number of rows,
and cubic Boolean bent functions.

%%%%%%%%%%%%%%%%%%%% Section 1 %%%%%%%%%%%%%%%%%%%%%%%%%%%%%%%%%%%%%%

\section{Functions}\label{BRECTS.Funcs}

Let $\ZZ_q$ be the ring of integers modulo~$q$ which 
we identify with the set $\{0,1,\ldots,q-1\}$.
Denote by~$\circup{\ZZ}_q$ the set of all $q$th roots of unity in~$\CC$
and introduce the additive character $\chi\colon\ZZ_q\to\circup{\ZZ}_q$, 
$\chi(a)=\exp(2\pi i a/q)$, where $i=\sqrt{-1}$. 
The ring of integers modulo a prime~$q$ is a field.
We emphasize this fact by writing~$\FF_q$ instead of $\ZZ_q$.

The set~$V_n=\ZZ_q^n$ consists of all vectors~$\vec{a}=(a_1,\ldots,a_n)$, 
$a_i\in\ZZ_q$, and forms a group under vector addition.
If~$q$ is prime, then~$V_n$ is 
the~$n$-dimensional vector space over the field~$\FF_q$.

Let~$\func{F}_n$ be the set of all functions $V_n\to\ZZ_q$.
A function~$f\in\func{F}_n$ depends on $n$ variables~---
coordinates of the argument~$\vec{x}=(x_1,\ldots,x_n)$.
Given $J=\{j_1,\ldots,j_m\}\subseteq\{1,\ldots,n\}$,
the restriction of~$f(\vec{x})$ to~$(x_{j_1},\ldots,x_{j_m})$
is a function obtained from~$f$ by keeping variables $x_j$, $j\not\in J$,
constant.

Starting from $f$, construct the function 
$\circup{f}\colon V_n\to\circup{\ZZ}_q$, 
$\circup{f}(\vec{x})=\chi(f(\vec{x}))$,
and the function $\hatup{f}\colon V_n\to\CC$,
$$
\hatup{f}(\vec{u})=
\sum_{\vec{x}\in V_n}
\circup{f}(\vec{x})
\overline{\chi(\vec{u}\cdot\vec{x})},\quad 
\vec{u}\in V_n.
$$
Here~$\vec{u}\cdot\vec{x}=u_1 x_1+\ldots+u_n x_n$
and the bar indicates complex conjugation.
It is convenient to assume that~$\func{F}_0$ consists 
of constant functions~$f\equiv c$, $c\in\ZZ_q$, 
for which~$\circup{f}=\hatup{f}\equiv\chi(c)$.
Denote $\circup{\func{F}}_n=\{\circup{f}\colon f\in\func{F}_n\}$, 
$\hatup{\func{F}}_n=\{\hatup{f}\colon f\in\func{F}_n\}$.

The conversion $\circup{f}\mapsto\hatup{f}$ (or $f\mapsto\hatup{f}$)
is called the {\it Walsh--Hadamard transform}.
Since 
$$
\sum_{\vec{u}\in V_n}\chi(\vec{a}\cdot\vec{u})=\left\{
\begin{array}{rl}
q^n, & \text{$\vec{a}=\vec{0}$},\\
0    & \text{otherwise}\\
\end{array}
\right.
$$
for each~$\vec{a}\in V_n$, 
the inverse transform~$\hatup{f}\mapsto\circup{f}$ can be defined as follows:
\begin{align*}
\circup{f}(\vec{x})&=
q^{-n}\sum_{\vec{y}\in V_n}\circup{f}(\vec{y})
\sum_{\vec{u}\in V_n}
\chi(\vec{u}\cdot(\vec{x}-\vec{y}))\\
&=q^{-n}\sum_{\vec{u}\in V_n}\sum_{\vec{y}\in V_n}
\circup{f}(\vec{y})
\overline{\chi(\vec{u}\cdot\vec{y})}
\chi(\vec{u}\cdot\vec{x})\\
&=q^{-n}\sum_{\vec{u}\in V_n}
\hatup{f}(\vec{u})\chi(\vec{u}\cdot\vec{x}).
\end{align*}

The Walsh--Hadamard transform is often used in cryptography,
coding theory, signal processing to derive some properties of~$f$.
In many cases it is important to use functions~$f$ with 
the characteristic~$\max_{\vec{u}}|\hatup{f}(\vec{u})|$
as small as possible.
Due to the Parseval's identity
$$
\sum_{\vec{u}\in V_n}
|\hatup{f}(\vec{u})|^2=
q^{2n},
$$
we have $\max_{\vec{u}}|\hatup{f}(\vec{u})|\geq q^{n/2}$,
where the equality holds if and only if
\begin{equation}\label{Eq.BRECT.BentDef}
|\hatup{f}(\vec{u})|=q^{n/2},\quad
\vec{u}\in V_n.
\end{equation}

If $f$ satisfies~\eqref{Eq.BRECT.BentDef}, 
then it is called a {\it bent} function.
If, additionally to~\eqref{Eq.BRECT.BentDef},
$\hatup{f}(\vec{u})\in q^{n/2}\circup{\ZZ}_q$ for all $\vec{u}$, 
then~$f$ is called a {\it regular} bent function.
Let~$\func{B}_n$ be the set of all regular bent functions of
$n$ variables. 

Further all subscripts, superscripts and other notations
defined for $\func{F}_n$ are automatically transferred 
to all subsets of $\func{F}_n$.
For example, 
$\circup{\func{B}}_n=\{\circup{f}\colon f\in\func{B}_n\}$,
$\hatup{\func{B}}_n=\{\hatup{f}\colon f\in\func{B}_n\}$.
Note that
$\hatup{\func{B}}_n=q^{n/2}\circup{\func{B}}_n$.

%%%%%%%%%%%%%%%%%%%% Section 2 %%%%%%%%%%%%%%%%%%%%%%%%%%%%%%%%%%%%%%

\section{Rectangles}\label{BRECTS.Rects}

Let~$m$, $k$ be nonnegative integers and~$f\in\func{F}_{m+k}$.
Define the function $\boxup{f}\colon V_{m+k}\to\CC$,
\begin{equation}\label{Eq.BRECT.RectFDef}
\boxup{f}(\vec{u},\vec{v})=
\sum_{\vec{y}\in V_k}
\circup{f}(\vec{u},\vec{y})
\overline{\chi(\vec{v}\cdot\vec{y})},\quad
\vec{u}\in V_m,\quad
\vec{v}\in V_k,
\end{equation}
and call it a {\it rectangle} of~$f$.
Denote by~$\boxup{\func{F}}_{m,k}$ the set of all such rectangles.
For the case $m=k$ we also call $\boxup{f}$ a {\it square} of~$f$.

To each function~$f$ there correspond different rectangles of the sets
$\boxup{\func{F}}_{0,m+k}=\hatup{\func{F}}_{m+k}$,
$\boxup{\func{F}}_{1,m+k-1},\ldots$,
$\boxup{\func{F}}_{m+k-1,1}$,
$\boxup{\func{F}}_{m+k,0}=\circup{\func{F}}_{m+k}$.
These rectangles are connected with each other.
For example, if~$\boxup{f}\in\boxup{\func{F}}_{m,k}$ 
and~$\boxup{f}^*\in\boxup{\func{F}}_{m+1,k-1}$, then
\begin{align*}
\boxup{f}(\vec{u},v,\vec{v}')&=
\sum_{y\in\ZZ_q}\sum_{\vec{y}'\in V_{k-1}}
\circup{f}(\vec{u},y,\vec{y}')
\overline{\chi(vy+\vec{v}'\cdot\vec{y}')}\\
&=
\sum_{y\in\ZZ_q}\boxup{f}^*(\vec{u},y,\vec{v}')\overline{\chi(vy)},\quad
\vec{v}'\in V_{k-1},\\
\intertext{and conversely,}
\boxup{f}^*(\vec{u},v,\vec{v}')&=
\frac{1}{q}\sum_{y\in\ZZ_q}\boxup{f}(\vec{u},y,\vec{v}')\chi(vy).
\end{align*}

For a fixed $\vec{u}$ call the
mapping~$\vec{v}\mapsto\boxup{f}(\vec{u},\vec{v})$ 
a {\it column} of~$\boxup{f}$.
Analogously, for a fixed $\vec{v}$ call the mapping
$\vec{u}\mapsto\boxup{f}(\vec{u},\vec{v})$
a {\it row} of~$\boxup{f}$.
By definition,
each row of~$\boxup{f}$ is an element of~$\hatup{\func{F}}_k$.
If furthermore each column of~$\boxup{f}$
multiplied by~$q^{(m-k)/2}$ is an element of $\hatup{\func{F}}_m$,
then call the rectangle~$\boxup{f}$ {\it bent}.

Our results are based on the following proposition,
first stated in~\cite{BRECTS.Agi02} for the case~$q=2$.
Note that under $m=0$ this proposition 
can be considered as a definition of regular bent functions.

%We can interpret bent rectangles as matrices with 
%restrictions on rows and columns
%and draw an analogy with latin rectangles.

\begin{proposition}\label{Prop.BRECTS.Idea}
A function~$f\in\func{F}_{m+k}$ is regular bent if and only if 
the rectangle $\boxup{f}\in\boxup{\func{F}}_{m,k}$ is bent.
\end{proposition}
\begin{proof}
Let $f\in\func{B}_n$, $n=m+k$.
Define the function $g\in\func{F}_n$ by the rule
$$
\circup{g}(\vec{v},\vec{u})=
q^{-n/2}\hatup{f}(-\vec{u},\vec{v}),\quad
\vec{u}\in V_m,\quad
\vec{v}\in V_k,
$$
and determine the corresponding rectangle
$\boxup{g}\in\boxup{\func{F}}_{k,m}$:
\begin{align*}
\boxup{g}(\vec{v},\vec{u})&=
q^{-n/2}\sum_{\vec{x}\in V_m}\hatup{f}(-\vec{x},\vec{v})
\overline{\chi(\vec{u}\cdot\vec{x})}\\
&=q^{-n/2}
\sum_{\vec{x}\in V_m}
\sum_{\vec{w}\in V_m}
\sum_{\vec{y}\in V_k}
\circup{f}(\vec{w},\vec{y})
\overline{\chi(-\vec{x}\cdot\vec{w}+
\vec{v}\cdot\vec{y}+\vec{u}\cdot\vec{x})}\\
&=q^{-n/2}
\sum_{\vec{y}\in V_k}
\sum_{\vec{w}\in V_m}
\circup{f}(\vec{w},\vec{y})
\overline{\chi(\vec{v}\cdot\vec{y})}
\sum_{\vec{x}\in V_m}
\chi((\vec{w}-\vec{u})\cdot\vec{x})\\
&=q^{m-n/2}
\sum_{\vec{y}\in V_k}
\circup{f}(\vec{u},\vec{y})
\overline{\chi(\vec{v}\cdot\vec{y})}\\
&=q^{(m-k)/2}\boxup{f}(\vec{u},\vec{v}).
\end{align*}
Therefore, each column of $\boxup{f}$ multiplied by $q^{(m-k)/2}$ is
an element of $\hatup{\func{F}}_m$ and $\boxup{f}$ is bent.

Conversely, if $\boxup{f}$ is bent, then 
$\boxup{g}(\vec{v},\vec{u})=q^{(m-k)/2}\boxup{f}(\vec{u},\vec{v})$
is well defined rectangle that corresponds to 
the function $\circup{g}(\vec{v},\vec{u})=q^{-n/2}\hatup{f}(\vec{u},\vec{v})$.
Hence $\hatup{f}(\vec{u},\vec{v})\in q^{n/2}\circup{\ZZ}_q$ 
for all~$\vec{u}$, $\vec{v}$ and~$f$ is regular bent.
\end{proof}

It is convenient to identify $\boxup{f}\in\boxup{\func{F}}_{m,k}$ 
with the $q^m\times q^k$ matrix $\boxup{F}$ whose
rows and columns are marked by lexicographically ordered
vectors of~$V_m$ and $V_k$ respectively and whose elements 
are the values $\boxup{f}(\vec{u},\vec{v})$.

The definition of~$\boxup{F}$ puts restrictions on its rows.
The bentness of $\boxup{F}$ puts additional restrictions on columns.
We can draw an analogy with latin rectangles
and this analogy justify the use of the term ``bent {\it rectangle}''.

If~$\boxup{F}$ corresponds to a bent rectangle 
$\boxup{f}\in\boxup{\func{B}}_{m,k}$ and satisfies restrictions on rows
and columns, then~$\boxup{G}=q^{(m-k)/2}\boxup{F}^{\rm T}$
also satisfies such restrictions and corresponds to
a bent rectangle~$\boxup{g}\in\boxup{\func{B}}_{k,m}$.
Call the transformation~$\boxup{f}\mapsto\boxup{g}$ 
a {\it transposition} of~$\boxup{f}$.
During the proof of Proposition~\ref{Prop.BRECTS.Idea} 
we actually showed that, under the transposition,
$$
\circup{g}(\vec{v},\vec{u})=q^{-n/2}\hatup{f}(-\vec{u},\vec{v}),\quad
\vec{u}\in V_m,\quad
\vec{v}\in V_k.
$$

As ``material'' for constructing bent rectangles 
we will often use affine functions $l(\vec{x})=\vec{b}\cdot\vec{x}+c$,
where $\vec{x},\vec{b}\in V_n$ and $c\in\ZZ_q$.
Denote by $\func{A}_n$ the set of all affine functions of~$n$ variables.

The function $\hatup{l}$ that corresponds to~$l$ 
has the quite simple form:
$$
\hatup{l}(\vec{u})=\left\{
\begin{array}{rl}
q^n\chi(c), & \text{if $\vec{u}=\vec{b}$},\\
0    & \text{otherwise}.
\end{array}
\right.
$$
This fact immediately gives us the following construction.

\begin{example}[Maiorana--McFarland's construction]\label{Ex.BRECTS.Maiorana}
Consider a bent square $\boxup{f}\in\boxup{\func{B}}_{n,n}$
such that all its rows and columns belong to~$\hatup{\func{A}}_n$.
The matrix~$\boxup{F}$ associated with~$\boxup{f}$ has the following form:
each its row and each column contains exactly one element 
of the set $q^n\circup{\mathbb Z}_q$, all other elements are zero. 
It means that there exists a permutation $\vec{\pi}\colon V_n\to V_n$ 
and a function $\varphi\in\func{F}_n$ such that
$$
\boxup{f}(\vec{u},\vec{v})=
\left\{
\begin{array}{rl}
q^n\circup{\varphi}(\vec{u}),& 
\text{if $\vec{v}=\vec{\pi}(\vec{u})$},\\
0 & \text{otherwise}.
\end{array}
\right.
$$
Consequently,
$f(\vec{x},\vec{y})=
\vec{\pi}(\vec{x})\cdot\vec{y}+\varphi(\vec{x})$
and we obtain the well-known Maiorana-McFarland's bent function.
\qed
\end{example}

%%%%%%%%%%%%%%%%%%%% Section 3 %%%%%%%%%%%%%%%%%%%%%%%%%%%%%%%%%%%%%%

\section{Affine transformations}\label{BRECTS.ATrans}

Let $\GL_n$ ($\AGL_n$) be the general linear (affine) group
of transformations of~$V_n=\ZZ_q^n$.
We identify~$\GL_n$ with the set of all invertible $n\times n$ 
matrices over~$\ZZ_q$ and denote by~$I_n$ the identity matrix of $\GL_n$.
An element~$\sigma\in\AGL_n$ is specified by a pair $(A,\vec{a})$,
$A\in\GL_n$, $\vec{a}\in V_n$, 
and acts as follows:
$\sigma(\vec{x})=\vec{x}A+\vec{a}$.
Extend the action of~$\AGL_n$ to functions~$f$ with domain~$V_n$ 
in a natural way:
$$
\sigma(f)(\vec{x})=f(\vec{x}A+\vec{a}).
$$

Call functions $f,g\in\func{F}_n$ {\it affine equivalent} 
if there exist $\sigma=(A,\vec{a})\in\AGL_n$ and 
$l\in\func{A}_n$, $l(\vec{x})=\vec{b}\cdot\vec{x}+c$, such that
\begin{equation}\label{Eq.BRECTS.AEquiv}
g(\vec{x})=\sigma(f)(\vec{x})+l(\vec{x})=
f(\vec{x}A+\vec{a})+\vec{b}\cdot\vec{x}+c.
\end{equation}

If~$f$ and~$g$ are affine equivalent, then
\begin{align*}
\hatup{g}(\vec{u})&=
\sum_{\vec{x}\in V_n}
\chi(f(\vec{x}A+\vec{a})+\vec{x}\cdot(\vec{b}-\vec{u})+c)\\
&=\sum_{\vec{y}\in V_n}
\chi(f(\vec{y})+(\vec{y}-\vec{a})A^{-1}\cdot(\vec{b}-\vec{u})+c)\\
&=\sum_{\vec{y}\in V_n}
\chi(f(\vec{y})-(\vec{y}-\vec{a})\cdot(\vec{u}-\vec{b})(A^{-1})^{\rm T}+c)
\end{align*}
and
\begin{equation}\label{Eq.BRECTS.AEquivSpectrum}
\hatup{g}(\vec{u})=\chi(\vec{a}\cdot\vec{v}+c)\hatup{f}(\vec{v}),\quad
\vec{v}=(\vec{u}-\vec{b})(A^{-1})^{\rm T}.
\end{equation}
Therefore, the functions $\hatup{f}$ and $\hatup{g}$ are, in some sense,
also affine equivalent: 
there exist $\sigma^*\in\AGL_n$ and $l^*\in\func{A}_n$ such that
$$
\hatup{g}(\vec{u})=\sigma^*(\hatup{f})(\vec{u})\circup{l}^*(\vec{u}).
$$

It is useful to describe a connection between 
rectangles~$\boxup{f},\boxup{g}\in\boxup{\func{F}}_{m,k}$
of affine equivalent functions.
For example, $g\in\func{B}_n$ if and only if $f\in\func{B}_n$
and a correspondence between $\boxup{f}$ and $\boxup{g}$ can be used
to perform the affine classification of regular bent functions.

Let $n=m+k$, where $m,k\geq 1$.
Divide the vector $\vec{x}\in V_n$ into two parts:
$\vec{x}=(\vec{x}_1,\vec{x}_2)$, 
$\vec{x}_1\in V_m$,
$\vec{x}_2\in V_k$,
and introduce the following {\it elementary transformations} $f\mapsto g$:
\begin{itemize}
\item[A1)]
$g(\vec{x})=f(\vec{x_1}A_1+\vec{a}_1,\vec{x}_2)$,
where $A_1\in\GL_m$, $\vec{a}_1\in V_m$;

\item[A2)]
$g(\vec{x})=f(\vec{x_1},\vec{x}_2 A_2)+\vec{a}_2\cdot\vec{x}_2$,
where $A_2\in\GL_k$, 
$\vec{a}_2\in V_k$;

\item[B1)]
$g(\vec{x})=f(\vec{x})+\vec{b}_1\cdot\vec{x}_1+c$,
where $\vec{b}_1\in V_m$, $c\in\ZZ_q$;

\item[B2)]
$g(\vec{x})=f(\vec{x_1},\vec{x}_2+\vec{b}_2)$,
where $\vec{b}_2\in V_k$;

\item[C1)]
$g(\vec{x})=f(x_1,\ldots,x_{m-1},x_{m}-x_{m+1},x_{m+1},x_{m+2}\ldots,x_{n})$;

\item[C2)]
$g(\vec{x})=f(x_1,\ldots,x_{m-1},x_m,x_{m+1}-x_m,x_{m+2},\ldots,x_{n})$.
\end{itemize}

\begin{proposition}\label{Prop.BRECTS.AEquiv}
Every affine transformation $f\mapsto g$ of the form~\eqref{Eq.BRECTS.AEquiv}
can be realized using only elementary transformations A1~-- C2.
Under these transformations, 
the rectangles $\boxup{f},\boxup{g}\in\boxup{\func{F}}_{m,k}$
are connected in the following manner:
\begin{itemize}
\item[A1)]
$\boxup{g}(\vec{u},\vec{v})=
\boxup{f}(\vec{u}A_1+\vec{a}_1,\vec{v})$,

\item[A2)]
$\boxup{g}(\vec{u},\vec{v})=
\boxup{f}(\vec{u},(\vec{v}-\vec{a}_2)B_2)$, 
where $B_2=(A_2^{-1})^{\rm T}$;

\item[B1)]
$\boxup{g}(\vec{u},\vec{v})=
\chi(\vec{b}_1\cdot\vec{u}+c)\boxup{f}(\vec{u},\vec{v})$;

\item[B2)]
$\boxup{g}(\vec{u},\vec{v})=
\chi(\vec{b}_2\cdot\vec{v})\boxup{f}(\vec{u},\vec{v})$;

\item[C1)]
$\displaystyle\boxup{g}(\vec{u}',u,v,\vec{v}')=
\frac{1}{q}\sum_{x,y\in\ZZ_q}
\boxup{f}(\vec{u}',x,y,\vec{v}')\overline{\chi((u-x)(v-y))}$;

\item[C2)]
$\boxup{g}(\vec{u}',u,v,\vec{v}')=
\boxup{f}(\vec{u}',u,v,\vec{v}')\overline{\chi(uv)}$.
\end{itemize}
Here $\vec{u}\in V_m$, $\vec{v}\in V_k$,
$\vec{u}'\in V_{m-1}$, $\vec{v}'\in V_{k-1}$,
$u,v\in\ZZ_q$.
\end{proposition}

\begin{proof}
To prove the first part of the proposition, we will establish 
that~$\GL_n$ is generated by the matrices of its subgroups
$$
G_1=\left\{
\begin{pmatrix}
M_1 & 0\\
0   & I_k
\end{pmatrix}\colon
M_1\in\GL_m
\right\},\quad
G_2=\left\{
\begin{pmatrix}
I_m & 0\\
0   & M_2
\end{pmatrix}\colon
M_2\in\GL_k
\right\}
$$
and additional matrices $R$ and $S$ of the following linear transformations:
\begin{align*}
\vec{x}&\mapsto
\vec{x}R=(x_1,\ldots,x_{m-1},x_{m}-x_{m+1},x_{m+1},x_{m+2},\ldots,x_{n}),\\
\vec{x}&\mapsto
\vec{x}S=(x_1,\ldots,x_{m-1},x_m,x_{m+1}-x_m,x_{m+2},\ldots,x_{n}).
\end{align*}

It is known (see, for example, \cite{BRECTS.Wae67}) 
that the group $\GL_n$ over the Euclidian ring~$\ZZ_m$
is generated by matrices of the following transformations 
of a vector~$\vec{x}$:
\begin{enumerate}
\item[1)]
multiplication of the coordinates~$x_i$ 
by invertible elements of $\ZZ_m$, $1\leq i\leq n$;
\item[2)]
subtraction of $x_j$ from $x_i$, $1\leq i,j\leq n$, $i\neq j$.
\end{enumerate}

The groups $G_1$ and $G_2$ contain matrices which realize
all transformations of the first type 
and transformations of the second type 
for $1\leq i,j\leq m$ and $m+1\leq i,j\leq n$.
We can subtract any different coordinates of $\vec{x}$ 
using additional matrices $R$ and $S$. 
For example, the subtraction of~$x_j$, $m+1\leq j\leq n$, 
from~$x_i$, $1\leq i\leq m$, can be realized by the following steps:
a)~interchange $x_i$ and $x_{m}$ using some matrix of~$G_1$,
b)~interchange $x_{m+1}$ and $x_j$ using some matrix of~$G_2$,
c)~subtract $x_{m+1}$ from $x_m$ using $R$, 
d)~interchange $x_i$ and $x_{m}$, $x_{m+1}$ and $x_j$ again.
Thus, the group generated by $G_1$, $G_2$, $R$, $S$
contains all matrices of the first and the second types and 
this group coincides with~$\GL_n$.

The second part of the proposition is checked by direct calculations.
Consider, for example, the transformation C1.
We have
\begin{align*}
\boxup{g}(\vec{u}',u,v,\vec{v}')
&=\sum_{x\in\ZZ_q}\sum_{\vec{y}'\in V_{k-1}}
\circup{g}(\vec{u}',u,x,\vec{y}')
\overline{\chi(vx+\vec{v}'\cdot\vec{y}')}\\
&=\sum_{x\in\ZZ_q}\sum_{\vec{y}'\in V_{k-1}}
\circup{f}(\vec{u}',u-x,x,\vec{y}')
\overline{\chi(vx+\vec{v}'\cdot\vec{y}')}\\
&=\sum_{x\in\ZZ_q}\sum_{\vec{y}'\in V_{k-1}}
\circup{f}(\vec{u}',x,u-x,\vec{y}')
\overline{\chi(v(u-x)+\vec{v}'\cdot\vec{y}')}\\
&=\frac{1}{q}\sum_{x,y,z\in\ZZ_q}\sum_{\vec{y}'\in V_{k-1}}
\circup{f}(\vec{u}',x,z,\vec{y}')
\overline{\chi((u-x)(v-y)+yz+\vec{v}'\cdot\vec{y}')}\\
&=\frac{1}{q}\sum_{x,y\in\ZZ_q}
\boxup{f}(\vec{u}',x,y,\vec{v}')\overline{\chi((u-x)(v-y))}
\end{align*}
and the required identity holds.
\end{proof}

Therefore, we can realize any affine transformation of 
a rectangle~$\boxup{F}$ by affine permutations of 
its rows (A1) and columns (A2) 
and by multiplying its elements by $q$th roots of unity
(B1, B2, C2). 
The remaining transformation C1 is a single method of changing
the elements of $\boxup{F}$ in magnitude.

\begin{example}[cells]\label{Ex.BRECTS.ATrans}
For the case $q=2$ it is convenient 
to illustrate the transformations~C1, C2 
in the following way.
Divide~$\boxup{F}$ into the {\it cells}, that is, the submatrices
$$
\begin{pmatrix}
\boxup{f}(\vec{u}',0,0,\vec{v}') & \boxup{f}(\vec{u}',0,1,\vec{v}')\\
\boxup{f}(\vec{u}',1,0,\vec{v}') & \boxup{f}(\vec{u}',1,1,\vec{v}')
\end{pmatrix},\quad
\vec{u}'\in V_{m-1},\quad
\vec{v}'\in V_{k-1},
$$
and during the transformations
modify all of the cells simultaneously by the rules
$$
{\rm C1}\colon
\begin{pmatrix}
\alpha&\beta\\
\gamma&\delta
\end{pmatrix}\mapsto
\frac{1}{2}
\begin{pmatrix}
\alpha+\beta+\gamma-\delta &\ \ \ \alpha+\beta-\gamma+\delta\\
\alpha-\beta+\gamma+\delta &\ \ \ -\alpha+\beta+\gamma+\delta\\
\end{pmatrix},\quad
{\rm C2}\colon
\begin{pmatrix}
\alpha&\beta\\
\gamma&\delta
\end{pmatrix}\mapsto
\begin{pmatrix}
\alpha&\beta\\
\gamma&-\delta
\end{pmatrix}.	
$$
\qed
\end{example}

%%%%%%%%%%%%%%%%%%%% Section 4 %%%%%%%%%%%%%%%%%%%%%%%%%%%%%%%%%%%%%%

\section{Illustrations}\label{BRECTS.Examples}

In this section we consider the case $q=2$ and
give examples to illustrate the 
usage of bent rectangles for the analysis of some known 
properties of bent functions.
Note that in the Boolean case each bent function is regular
and $\func{B}_n\neq\varnothing$ only for even~$n$.

Start with some useful definitions and facts.
Firstly, 
identify a Boolean function~$f\in\func{F}_n$ of~$\vec{x}=(x_1,\ldots,x_n)$ 
with its algebraic normal form, that is, 
a polynomial of the ring~$\FF_2[x_1,\ldots,x_n]$
reduced modulo the ideal~$(x_1^2-x_1,\ldots,x_n^2-x_n)$.
Denote by $\deg f$ the 
degree of such a polynomial.

Following~\cite{BRECTS.ZheZha99}, 
introduce the set $\func{P}_{n,r}\subseteq\func{F}_n$ 
of plateaued functions of order $r$:
$f\in\func{P}_{n,r}$ if $|\hatup{f}(\vec{u})|\in\{0,2^{n-r/2}\}$ 
for all $\vec{u}\in V_n$ 
(more precisely, $\hatup{f}$ has exactly $2^r$ nonzero values $\pm 2^{n-r/2}$).
It is clear that
$\func{A}_n=\func{P}_{n,0}$,
$\func{B}_n=\func{P}_{n,n}$.

Finally, recall the following result of~\cite{BRECTS.Agi02}.

\begin{lemma}\label{Lemma.BRECTS.Quartet}
Let~$q=2$, $f_1, f_2, f_3, f_4\in\func{F}_n$ and
$$
\hatup{g}(\vec{u})=\frac{1}{2}(\hatup{f}_1(\vec{u})+
\hatup{f}_2(\vec{u})+\hatup{f}_3(\vec{u})+\hatup{f}_4(\vec{u})),\quad
\vec{u}\in V_n.
$$
The function $\hatup{g}\in\hatup{\func{F}}_n$ if and only if
$$
f_1(\vec{x})+f_2(\vec{x})+f_3(\vec{x})+f_4(\vec{x})=1,\quad
\vec{x}\in V_n.
$$
Under this condition,
$g(\vec{x})=f_1(\vec{x})f_2(\vec{x})+
f_1(\vec{x})f_3(\vec{x})+
f_2(\vec{x})f_3(\vec{x})$.
\end{lemma}

Turn to examples.

\begin{example}[sums of bent functions]\label{Ex.BRECTS.Sum}
It is well-known (see~\cite{BRECTS.Rot76}) that 
if~$f_1\in\func{B}_m$ and~$f_2\in\func{B}_k$,
then $f(\vec{x},\vec{y})=f_1(\vec{x})+f_2(\vec{y})$ is also bent.
Indeed, the rectangle~$\boxup{f}\in\boxup{\func{F}}_{m,k}$ 
has the form
$\boxup{f}(\vec{u},\vec{v})=\circup{f}_1(\vec{u})\hatup{f}_2(\vec{v})$
and obviously satisfies the restrictions on columns.
\qed
\end{example}

\begin{example}[degrees of bent functions]\label{Ex.BRECTS.Deg}
Let~$f\in\func{B}_{2n}$, $n\geq 2$.
In~\cite{BRECTS.Rot76}, Rothaus proved that $\deg f\leq n$.
Give an alternative proof of this fact.

Suppose to the contrary that~$\deg f=k>n$, i.e.
the polynomial $f(x_1,\ldots,x_{2n})$ contains a monomial of degree~$k$.
Without loss of generality, assume that $f$ 
contains the monomial $x_{m+1}x_{m+2}\ldots x_{2n}$, $m=2n-k$.
Then there are an odd number of $1$'s among the values 
$g(\vec{y})=f(\vec{0},\vec{y})$, $\vec{y}\in V_k$,
and~$\hatup{g}(\vec{0})=2r$, where $r$ is odd.
It is impossible when $k=2n$, 
since in this case $\hatup{g}(\vec{0})=\hatup{f}(\vec{0})\in\{\pm 2^n\}$.
Therefore $k<2n$ and consequently $m>0$.
Consider~$\boxup{f}\in\boxup{\func{B}}_{m,k}$.
We have~$\boxup{f}(\vec{0},\vec{0})=\hatup{g}(\vec{0})=2r$
and~$2^{(m-k)/2}\boxup{f}(\vec{0},\vec{0})$ is not even integer.
Since functions of~$\hatup{\func{F}}_m$ take only even values,
the restrictions on columns of~$\boxup{f}$ are not satisfied,
a contradiction.
\qed
\end{example}

\begin{example}[$2$-row bent rectangles]\label{Ex.BRECTS.2Row}
Let $\boxup{f}\in\boxup{\func{B}}_{1,n-1}$.
By definition, the columns of~$\boxup{f}$ belong to the set 
$2^{n/2-1}\func{F}_1$. 
Since $\func{F}_1=\func{A}_1$, 
each column takes exactly one nonzero value $\pm 2^{n/2}$.
It means that all rows of $\boxup{f}$ are in~$\hatup{\func{P}}_{n-1,n-2}$.
This fact was pointed out in~\cite{BRECTS.ZheZha99}.

In Proposition~\ref{Prop.BRECTS.APartConstr} we will describe how to construct
$(2^n-2)|\func{B}_{n-2}|^2$ bent rectangles of~$\boxup{\func{B}}_{1,n-1}$.
For example, such a construction allows to obtain all 896 elements
of $\boxup{\func{B}}_{1,3}$.
\qed
\end{example}

\begin{example}[$4$-row bent rectangles]\label{Ex.BRECTS.4Row}
Consider a rectangle~$\boxup{f}\in\boxup{\func{B}}_{2,n-2}$.
Since every function of $\func{F}_2$ is either affine or bent,
the possible values of $\boxup{f}$ are exhausted by
$0$, $\pm 2^{n/2-1}$, $\pm 2^{n/2}$.

If we restrict to the numbers $\pm 2^{n/2-1}$, 
then each column of $\boxup{f}$ belongs to $2^{n/2-1}\circup{\func{B}}_2$,
takes an odd number of negative values, 
and a product of these values is exactly~$-2^{2n-4}$.
This result was proved in~\cite{BRECTS.PreVan91}.

If we restrict to the numbers $0$, $\pm 2^{n/2}$,
then we obtain a bent rectangle all rows of which 
are in $\hatup{\func{P}}_{n-2,n-4}$.
Proposition~\ref{Prop.BRECTS.APartConstr} will give
$8(2^{n-2}-1)(2^{n-3}-1)(7\cdot 2^{n-3}-13)
|\func{B}_{n-4}|^4$ such rectangles for~$n\geq 4$.
\qed
\end{example}

\begin{example}[Rothaus' construction]\label{Ex.BRECTS.Rothaus}
Let $f_1,f_2,f_3,f_4\in\func{B}_n$ satisfy
$f_1(\vec{y})+f_2(\vec{y})+f_3(\vec{y})+f_4(\vec{y})=0$
for all $\vec{y}\in V_n$. 
In~\cite{BRECTS.Rot76}, Rothaus showed that the function 
\begin{align*}
f(u_1,u_2,\vec{y})&=
f_1(\vec{y})f_2(\vec{y})+
f_1(\vec{y})f_3(\vec{y})+
f_2(\vec{y})f_3(\vec{y})\\
&+
u_1(f_1(\vec{y})+f_2(\vec{y}))+
u_2(f_1(\vec{y})+f_3(\vec{y}))+u_1 u_2
\end{align*}
is bent.
Give another proof of this fact using the rectangle
$\boxup{f}\in\boxup{\func{F}}_{2,n}$.

Lemma~\ref{Lemma.BRECTS.Quartet} implies that $\boxup{f}$ has
the following form:
\begin{equation}\label{Eq.BRECTS.Rothaus}
\boxup{f}(u_1,u_2,\vec{v})=
\frac{1}{2}\left\{
\begin{array}{rl}
\hatup{f}_1(\vec{v})+\hatup{f}_2(\vec{v})+
\hatup{f}_3(\vec{v})-\hatup{f}_4(\vec{v}), & u_1=u_2=0,\\
\hatup{f}_1(\vec{v})-\hatup{f}_2(\vec{v})+
\hatup{f}_3(\vec{v})+\hatup{f}_4(\vec{v}), & u_1=0,\  u_2=1,\\
\hatup{f}_1(\vec{v})+\hatup{f}_2(\vec{v})-
\hatup{f}_3(\vec{v})+\hatup{f}_4(\vec{v}), & u_1=1,\  u_2=0,\\
\hatup{f}_1(\vec{v})-\hatup{f}_2(\vec{v})-
\hatup{f}_3(\vec{v})-\hatup{f}_4(\vec{v}), & u_1=u_2=1.
\end{array}
\right.
\end{equation} 
Choose an arbitrary $\vec{v}$ and construct the function $h\in\func{F}_2$:
$\circup{h}(0,0)=2^{-n/2}\hatup{f}_1(\vec{v})$,
$\circup{h}(0,1)=2^{-n/2}\hatup{f}_2(\vec{v})$,
$\circup{h}(1,0)=2^{-n/2}\hatup{f}_3(\vec{v})$,
$\circup{h}(1,1)=-2^{-n/2}\hatup{f}_4(\vec{v})$.
By~\eqref{Eq.BRECTS.Rothaus}, 
the $\vec{v}$th column of $\boxup{f}$ coincides with $2^{n/2-1}\hatup{h}$.
Hence $\boxup{f}$ satisfies restrictions on columns and $f$ is bent.

Observe that the function $\boxup{f}(u_1,u_2,\vec{y})$ can take all the values 
$0$, $\pm 2^{n/2-1}$, $\pm 2^{n/2}$ and the Rothaus' construction
provides more subtle 4-row rectangles 
than ones considered in the previous example. 
\qed
\end{example}

\begin{example}[Carlet's transformation]\label{Ex.BRECTS.Carlet}
Let $f\in\func{B}_{2n}$,
$E\subset V_{2n}$ be an affine plane of dimension $k\geq n$,
and $\phi_{E}\in\func{F}_{2n}$ be the support of $E$, that is, 
$\phi_{E}(\vec{x})=1$ if only if $\vec{x}\in E$.
In~\cite[p.~94]{BRECTS.Car94}, Carlet obtained conditions
which provide the bentness of the function
$g(\vec{x})=f(\vec{x})+\phi_{E}(\vec{x})$.
Let us prove yet another condition:
$g\in\func{B}_{2n}$ if and only if $f_E\in\func{P}_{k,2(k-n)}$,
where $f_E$ is a restriction of $f$ to $E$, that is,
$f_E(\vec{y})=f(\varphi^{-1}(\vec{y}))$ for some affine bijection
$\varphi\colon E\to V_k$. 

Without loss of generality, 
assume that $E=\{(\vec{0},\vec{a})\colon\vec{a}\in V_k\}$
and $\varphi\colon(\vec{0},\vec{a})\mapsto\vec{a}$.
Consider the rectangles $\boxup{f}\in\boxup{\func{B}}_{2n-k,k}$ and
$\boxup{g}\in\boxup{\func{F}}_{2n-k,k}$.
They differ only in the first row:
$$
\boxup{f}(\vec{0},\vec{v})=\hatup{f}_E(\vec{v}),\quad
\boxup{g}(\vec{0},\vec{v})=-\hatup{f}_E(\vec{v}).
$$
Let $\hatup{f}_E(\vec{v})\neq 0$ for some~$\vec{v}$.
Consider the normalized columns
$\hatup{h}(\vec{u})=2^{n-k}\boxup{f}(\vec{u},\vec{v})$ and
$\hatup{h}'(\vec{u})=2^{n-k}\boxup{g}(\vec{u},\vec{v})$.
The functions $\hatup{h}\in\hatup{\func{F}}_{2n-k}$ and $\hatup{h}'$ differ
only in signs of their (nonzero) values at $\vec{u}=\vec{0}$.
Therefore, $\hatup{h}'\in\hatup{\func{F}}_{2n-k}$ 
if and only if $\hatup{h}(\vec{0})=\pm 2^{2n-k}$.
Under this condition,
$$                                 
\hatup{f}_E(\vec{v})=
\boxup{f}(\vec{0},\vec{v})=
2^{k-n}\hatup{h}(\vec{0})=\pm 2^n.
$$
Consequently, $\boxup{g}$ is bent if and only if
$\hatup{f}_E$ takes only the values $0$, $\pm 2^n$, 
that is, $f_E\in\func{P}_{k,2(k-n)}$.
\qed
\end{example}

\begin{example}[normal bent functions]
Following~\cite{BRECTS.Dob95}, call a function~$f\in\func{B}_{2n}$ 
(weakly) normal if its restriction to some $n$-dimesional affine plane
is affine.
In other words, $f$ is normal if there exists an affine equivalent 
function $g$ such that
$$
M(g)=\max_{\vec{u},\vec{v}\in V_n}|\boxup{g}(\vec{u},\vec{v})|=2^n,\quad
\boxup{g}\in\boxup{\func{B}}_{n,n}.
$$

Using the affine classification of bent functions of $6$ variables 
(see~\cite{BRECTS.Rot76}), one can check that every $f\in\func{B}_6$ is normal.
We give a direct proof of this fact by applying the results 
of Section~\ref{BRECTS.ATrans}.

Write $\varphi\sim\alpha_1^{d_1}\alpha_2^{d_2}\ldots$ for a function $\varphi$
that takes $d_1$ values $\alpha_1$, $d_2$ values $\alpha_2$, and so on.
If $\hatup{h}\in\hatup{\func{F}}_3$, then
$|\hatup{h}|\sim 8^1 0^7$ or $|\hatup{h}|\sim 6^1 2^7$ or
$|\hatup{h}|\sim 4^4 0^4$.
Since all rows and columns of $\boxup{f}\in\boxup{\func{B}}_{3,3}$ 
are in $\hatup{\func{F}}_3$, we have $M(f)\in\{4,6,8\}$.

Let $M(f)=4$. 
Applying the elementary transformations A1, A2 
(i.e. permuting rows and columns) 
and B1, B2 (i.e. changing signs of rows and columns), 
we can arrange elements of~$\boxup{f}$ into one of the following cells:
$$
\begin{pmatrix}
4&4\\
4&0\\
\end{pmatrix},\quad
\begin{pmatrix}
4&4\\
4&-4\\
\end{pmatrix},\quad
\begin{pmatrix}
4&4\\
4&4\\
\end{pmatrix}.
$$
Then apply C1, C2:
$$
\begin{pmatrix}
4&4\\
4&0\\
\end{pmatrix}\stackrel{\rm C1}{\to}
\begin{pmatrix}
6&2\\
2&2\\
\end{pmatrix},\quad
\begin{pmatrix}
4&4\\
4&-4\\
\end{pmatrix}\stackrel{\rm C1}{\to}
\begin{pmatrix}
8&0\\
0&0\\
\end{pmatrix},\quad
\begin{pmatrix}
4&4\\
4&4\\
\end{pmatrix}\stackrel{\rm C2}{\to}
\begin{pmatrix}
4&4\\
4&-4\\
\end{pmatrix}\stackrel{\rm C1}{\to}
\begin{pmatrix}
8&0\\
0&0\\
\end{pmatrix},
$$
and obtain a rectangle that takes the value $\geq 6$.
So it is sufficient to treat the case $M(f)=6$ only.
In this case we proceed in a similar manner:
$$
\xrightarrow{{\rm A1}, {\rm A2}, {\rm B1}, {\rm B2}, {\rm C2}}
\begin{pmatrix}
6&2\\
2&-6\\
\end{pmatrix}\stackrel{\rm C1}{\to}
\begin{pmatrix}
8&0\\
0&-4\\
\end{pmatrix},
$$
and obtain a rectangle that takes the value $8$,
which is what we wanted.
\qed
\end{example}

\begin{example}[the number of bent functions]
If $g\in\func{P}_{n,2}$, that is, $\hatup{g}(\vec{a}_i)=2^{n-1}\chi(b_i)$ 
for some $b_i\in\FF_2$ and distinct~$\vec{a}_i\in V_n$, $i=1,2,3,4$,
then by Lemma~\ref{Lemma.BRECTS.Quartet}
\begin{itemize}
\item[(i)]
$\vec{a}_1+\vec{a}_2+\vec{a}_3+\vec{a}_4=\vec{0}$,
\item[(ii)]
$b_1+b_2+b_3+b_4=1$.
\end{itemize}

In~\cite{BRECTS.Agi02}, we proposed an algorithm to construct bent squares
$\boxup{f}\in\boxup{\func{B}}_{n,n}$ all rows and columns of which
are elements of $\hatup{\func{A}}_n\cup\hatup{\func{P}}_{n,2}$.
The algorithm arranges nonzero elements in the matrix
$\boxup{F}$ subject to the condition~(i) 
and then places signs of elements subject to~(ii).

Counting different outputs of the algorithm, 
we obtain constructive lower bounds for~$|\boxup{\func{B}}_{n,n}|$.
In particular, the algorithm provides 
$1559994535674013286400> 2^{70.4}$ distinct bent squares of
$\boxup{\func{B}}_{4,4}$ and, consequently, $|\func{B}_8|> 2^{70.4}$.
\qed
\end{example}

%%%%%%%%%%%%%%%%%%%% Section 5 %%%%%%%%%%%%%%%%%%%%%%%%%%%%%%%%%%%%%%

\section{Biaffine and bilinear bent squares}\label{BRECTS.Bi}

Consider a mapping $\pi\colon V_n\times V_n\to V_n$.
As previously, for a fixed $\vec{v}$ call the mapping 
$\vec{u}\mapsto\pi(\vec{u},\vec{v})$ a restriction of $\pi$ to $\vec{u}$,
and for a fixed $\vec{u}$ call the mapping 
$\vec{v}\mapsto\pi(\vec{u},\vec{v})$ a restriction to $\vec{v}$.
Say that $\pi$ is {\it biaffine} ({\it bilinear})
if all its restrictions to $\vec{u}$ and $\vec{v}$ 
are affine (linear) transformations of~$V_n$.
A biaffine mapping is {\it nonsingular}
if all its restrictions are invertible.
A bilinear mapping is {\it nonsingular} if all its restrictions 
to $\vec{u}$ for $\vec{v}\neq\vec{0}$ and
to $\vec{v}$ for $\vec{u}\neq\vec{0}$ are invertible.

Choose an arbitrary function $g\in\func{F}_n$ 
and construct a bent square $\boxup{f}\in\boxup{\func{B}}_{n,n}$
such that the corresponding matrix $\boxup{F}$
consists almost only of the values of~$\hatup{g}$ 
permuted in some order and multiplied by elements of~$\circup{\ZZ}_q$.
Informally speaking, we ``scatter'' values of $\hatup{g}$ 
over~$\boxup{F}$.
The connection~\eqref{Eq.BRECTS.AEquivSpectrum} 
between the Walsh--Hadamard coefficients of affine equivalent functions
allows us to make such a scattering using a nonsingular 
biaffine or bilinear mapping.
The obtained constructions are given by the following two 
easily verified propositions.

\begin{proposition}[biaffine bent squares]\label{Prop.BRECTS.Biaffine}
Let $\pi$ be a nonsingular biaffine mapping $V_n\times V_n\to V_n$,
$g\in\func{F}_n$, $\varphi\in\func{F}_{2n}$, 
and all restrictions of $\varphi(\vec{u},\vec{v})$ 
to $\vec{u}\in V_n$ and $\vec{v}\in V_n$ are affine functions.
Then the square
$$
\boxup{f}(\vec{u},\vec{v})=
\circup{\varphi}(\vec{u},\vec{v})
\hatup{g}(\pi(\vec{u},\vec{v}))
$$
is bent.
\end{proposition}

\begin{proposition}[bilinear bent squares]\label{Prop.BRECTS.Bilinear}
Let $\pi$ be a nonsingular bilinear mapping $V_n\times V_n\to V_n$,
$g\in\func{F}_n$, 
$\varphi$ be defined as in the previous proposition,
and $h,h'\in\func{F}_n$ be such that
$$
\hatup{h}(\vec{0})=\hatup{h}'(\vec{0}),\quad
\hatup{h}(\vec{u})=\circup{\varphi}(\vec{u},\vec{0})\hatup{g}(\vec{0}),\quad
\hatup{h}'(\vec{v})=\circup{\varphi}(\vec{0},\vec{v})\hatup{g}(\vec{0}),\quad
\vec{u},\vec{v}\in V_n\setminus\{\vec{0}\}.
$$
Then the square
$$
\boxup{f}(\vec{u},\vec{v})=\left\{
\begin{array}{rl}
\hatup{h}(\vec{0}), & \vec{u}=\vec{v}=\vec{0},\\
\circup{\varphi}(\vec{u},\vec{v})\hatup{g}(\pi(\vec{u},\vec{v})) & 
\text{otherwise},
\end{array}
\right.
$$
is bent.
\end{proposition}

Note that any function of the form
$$ 
\varphi(u_1,\ldots,u_n,v_1,\ldots,v_n)=
\sum_{i,j=1}^n\alpha_{ij} u_i v_j+
\sum_{i=1}^n\beta_i u_i+
\sum_{i=1}^n\gamma_i v_i+\delta,\ \ 
\alpha_{ij},\beta_i,\gamma_j,\delta\in\ZZ_q,
$$
satisfies the condition of the above propositions.
Note also that in Proposition~\ref{Prop.BRECTS.Bilinear}
we can easily construct $h$ and $h'$ if $\hatup{g}(\vec{0})=0$. 
Indeed, in this case it is sufficiently to choose $h=h'\equiv c$, $c\in\ZZ_q$.

In Proposition~\ref{Prop.BRECTS.Biaffine}
we can use the following nonsingular biaffine mapping
$$ 
\pi(\vec{u},\vec{v})=\vec{u}A+\vec{v}B+
(\vec{u}C_1\vec{v}^{\rm T},\ldots,\vec{u}C_n\vec{v}^{\rm T})+
\vec{d},
$$
where $\vec{d}\in V_n$ and $A$, $B$, $C_1,\ldots,C_n$ are $n\times n$
matrices over $\ZZ_q$ such that 
$A+(C_1\vec{v}^{\rm T},\ldots,C_n\vec{v}^{\rm T})$ and
$B+(C_1^{\rm T}\vec{u}^{\rm T},\ldots,C_n^{\rm T}\vec{u}^{\rm T})$ 
are invertible for every $\vec{v}$ and $\vec{u}$.

Further, each nonsingular bilinear mapping $\pi$ has the form
\begin{equation}\label{Eq.BRECTS.Bilinear}
\pi(\vec{u},\vec{v})=\vec{u}A_{\vec{v}},
\end{equation}
where $A_{\vec{v}}$ are $n\times n$ matrices over $\ZZ_q$ such that
\begin{itemize}
\item[(i)]
$A_{\vec{0}}=0$ and
$A_{\vec{v}}\in\GL_n$ for all nonzero $\vec{v}\in V_n$;

\item[(ii)]
$A_{\vec{v}+\vec{v}'}=A_{\vec{v}}+A_{\vec{v}'}$
for all $\vec{v},\vec{v}'\in V_n$.
\end{itemize}

Indeed, considering restrictions of $\pi$ to $\vec{u}$,
we get~\eqref{Eq.BRECTS.Bilinear}.
Since each such restriction for $\vec{v}\neq\vec{0}$ must be invertible, 
$A_{\vec{v}}\in\GL_n$ for all nonzero $\vec{v}$.
For each fixed $\vec{u}$ the mapping
$\vec{v}\mapsto\vec{u}A_{\vec{v}}$ must be linear.
Therefore,
$\vec{u}A_{\vec{v}+\vec{v'}}=\vec{u}(A_{\vec{v}}+A_{\vec{v}'})$
for all $\vec{v}$ and $\vec{v}'$,
which yields~(ii) and the first part of~(i).

Assume further that $q$ is prime.
The arising structure~$R=\{A_{\vec{v}}\colon\vec{v}\in V_n\}$
is connected with some concepts of projective geometry.
Consider the following linear subspaces of $V_{2n}$:
\begin{equation}\label{Eq.BRECTS.Spread}
E_{\infty}=\{(\vec{0},\vec{v})\colon \vec{v}\in V_n\},\quad
E_{\vec{v}}=\{(\vec{u},\vec{u}A_{\vec{v}})\colon
\vec{u}\in V_n\},\quad\vec{v}\in V_n.
\end{equation}
Each such subspace is of dimension $n$,
two different subspaces intersect only in the zero vector
and, consequently, the union of subspaces is~$V_{2n}$.
A set of subspaces with these properties is called
a {\it spread} of $V_{2n}$.

Observe that the condition~(ii) 
is too strong for the set~\eqref{Eq.BRECTS.Spread} to be a spread.
We need only that 
$A-A'\in\GL_n$ for all distinct $A,A'\in R$
(with the additional property $I_n\in R$ such a set $R$
is called a quasifield).

Spreads are used in the following 
well-known construction of bent functions.

\begin{example}[Dillon's construction]
Consider a spread~\eqref{Eq.BRECTS.Spread} determined
by a set~$R=\{A_{\vec{v}}\colon\vec{v}\in V_n\}$.
Choose $c\in\FF_q$ and $g\in\func{F}_n$ such that $\hatup{g}(\vec{0})=0$.
In~\cite{BRECTS.Dil72}, Dillon actually proved 
that the function $f\in\func{F}_{2n}$,
$$
f(\vec{x})=\left\{
\begin{array}{rl}
c,          & \vec{x}\in E_{\infty},\\
g(\vec{v}), & \vec{x}\in E_\vec{v}\setminus\{\vec{0}\},
\end{array}
\right.
$$
is regular bent.

The corresponding bent square~$\boxup{f}\in\boxup{\func{B}}_{n,n}$
has the form
$$
\boxup{f}(\vec{u},\vec{v})=
\left\{
\begin{array}{rl}
q^n\chi(c), & \vec{u}=\vec{v}=\vec{0},\\
\displaystyle
\sum_{\vec{y}\in V_n}
\circup{g}(\vec{y})\overline{\chi(\vec{u}A_{\vec{y}}\cdot\vec{v})} &
\text{otherwise}.
\end{array}
\right.
$$
We see that the Dillon's bent squares are similar to the bilinear ones.
In particular, if elements of~$R$ satisfy the conditions~(i) and~(ii),
then the Dillon's construction is covered by 
Proposition~\ref{Prop.BRECTS.Biaffine}
under the choice $h=h'\equiv c$, $\varphi\equiv 0$,
$\pi(\vec{u},\vec{v})=\vec{v}B_{\vec{u}}^{\rm T}$,
where the matrices $B_{\vec{u}}$ are such that
$\vec{u}A_{\vec{y}}=\vec{y}B_{\vec{u}}$ for all~$\vec{u},\vec{y}\in V_n$.
\qed
\end{example}

Remark that the Maiorana--McFarland's bent squares can be represented
in the form of Proposition~\ref{Prop.BRECTS.Bilinear} 
with only the refinement that $g$ is necessarily affine.
The simple form of $\hatup{g}$ in this case allows
to relax the conditions of the proposition:
we can use an arbitrary $\varphi$ and choose $\pi(\vec{u},\vec{v})$ 
such that all its restrictions to $\vec{u}$ and $\vec{v}$
are bijections (possibly not affine).

%%%%%%%%%%%%%%%%%%%% Section 6 %%%%%%%%%%%%%%%%%%%%%%%%%%%%%%%%%%%%%%

\section{Partitions into affine planes}\label{BRECTS.APart}

Let~$q$ be prime and $L$ be a linear subspace of~$V_n=\FF_q^n$ 
of dimension~$r$. Write the latter as $L<V_n$, $\dim L=r$.
Recall that the number of distinct $r$-dimensional 
subspaces of~$V_n$ is given by the Gaussian coefficient
$$
\GaussCoeff{n}{r}{q}=
\frac{(q^n-1)(q^{n-1}-1)\ldots(q^{n-r+1}-1)}
{(q^r-1)(q^{r-1}-1)\ldots(q-1)},\quad
0\leq r\leq n.
$$

Let~$E=L+\vec{b}$ be an affine plane
obtained by a shift of~$L$ by a vector~$\vec{b}\in V_n$.
The plane~$E$ is the image of the affine mapping~$\pi\colon V_r\to V_n$,
$\vec{w}\mapsto\vec{w}A+\vec{b}$,
under a suitable choice of the $r\times n$ matrix $A$ over~$\FF_q$ of rank~$r$.

Given~$g\in\func{F}_r$, construct the function $h\in\func{F}_n$,
\begin{equation}\label{Eq.BRECTS.StretchFunc}
h(\vec{x})=
g(\vec{x}A^{\rm T})+\vec{b}\cdot\vec{x},\quad
\vec{x}\in V_n.
\end{equation}
Call the transformation $g\mapsto h$ a {\it stretching} of~$g$ to
the plane~$E$.
Under the stretching
\begin{align*}
\hatup{h}(\vec{v})&=
\sum_{\vec{x}\in V_n}\circup{g}(\vec{x}A^{\rm T})
\chi((\vec{b}-\vec{v})\cdot\vec{x})\\
&=q^{-r}\sum_{\vec{x}\in V_n}\sum_{\vec{w}\in V_r}
\hatup{g}(\vec{w})
\chi(\vec{x}A^{\rm T}\cdot\vec{w}+(\vec{b}-\vec{v})\cdot\vec{x})\\
&=q^{-r}\sum_{\vec{w}\in V_r}\hatup{g}(\vec{w})
\sum_{\vec{x}\in V_n}\chi((\pi(\vec{w})-\vec{v})\cdot\vec{x})
\end{align*}
and, consequently,
\begin{equation}\label{Eq.BRECTS.StretchSpectrum}
\hatup{h}(\vec{v})=\left\{
\begin{array}{rl}
q^{n-r}\hatup{g}(\pi^{-1}(\vec{v})), & \vec{v}\in E,\\
0 & \text{otherwise}.
\end{array}
\right.
\end{equation}

Describe bent rectangles $\boxup{f}\in\boxup{\func{B}}_{m,n}$, $m=n-r$,
which have all rows of the form~\eqref{Eq.BRECTS.StretchSpectrum}
and all columns in the set~$q^{r/2}\hatup{\func{A}}_m$.
Such bent rectangles were first introduced by 
Carlet~\cite{BRECTS.Car04} for the case $q=2$.
Note that after the transposition of $\boxup{f}$ we obtain
the rectangle~$\boxup{f}^*\in\boxup{\func{B}}_{n,m}$
that corresponds to the so called (see~\cite{BRECTS.LogSalYas00})
{\it partial affine} bent function $f^*\in\func{B}_{n+m}$:
each restriction of $f^*(\vec{x},\vec{y})$ to~$\vec{x}\in V_n$ is affine.

\begin{proposition}[partitions into affine planes]
\label{Prop.BRECTS.APartConstr}
Let $m$, $r$ be nonnegative integers, $n=m+r$,
$\vec{u}\in V_m$, $\vec{v}\in V_n$, and
$$
\boxup{f}(\vec{u},\vec{v})=\left\{
\begin{array}{rl}
q^m \hatup{g}_{\vec{u}}(\pi_{\vec{u}}^{-1}(\vec{v})), &
\vec{v}\in E_{\vec{u}},\\
0 & \text{otherwise},
\end{array}
\right.
$$
where $g_{\vec{u}}\in\func{B}_r$ and~$\pi_{\vec{u}}$ are 
mappings~$V_r\to V_n$ such that $E_{\vec{u}}=\pi_{\vec{u}}(V_r)$
are affine planes of dimension~$r$.
If the planes~$\{E_{\vec{u}}\}$ are disjoint
and therefore determine a partition of~$V_n$, 
then~$\boxup{f}\in{\boxup{\func{B}}}_{m,n}$.
\end{proposition}

\begin{example}\label{Ex.BRECTS.APartConstr}
Let $q=2$.
Choose the following partition of~$V_4$ 
into the planes of dimension~$2$:
\begin{align*}
E_{(0,0)}&=\{(0,0,0,0), (0,0,0,1), (0,0,1,0), (0,0,1,1)\},\\
E_{(0,1)}&=\{(0,1,0,0), (0,1,0,1), (1,1,0,0), (1,1,0,1)\},\\
E_{(1,0)}&=\{(0,1,1,0), (0,1,1,1), (1,0,0,0), (1,0,0,1)\},\\
E_{(1,1)}&=\{(1,0,1,0), (1,0,1,1), (1,1,1,0), (1,1,1,1)\}.
\end{align*}
Using this partition, construct
a rectangle $\boxup{f}\in\boxup{\func{B}}_{2,4}$ 
with the matrix
$$
\boxup{F}=
\left(
\begin{array}{cccccccccccccccc}
\pm 8 & \pm 8 & \pm 8 & \pm 8 & 0 & 0 & 0 & 0 & 0 & 0 & 0 & 0 & 0 & 0 & 0 & 0\\
0 & 0 & 0 & 0 & \pm 8 & \pm 8 & 0 & 0 & 0 & 0 & 0 & 0 & \pm 8 & \pm 8 & 0 & 0\\
0 & 0 & 0 & 0 & 0 & 0 & \pm 8 & \pm 8 & \pm 8 & \pm 8 & 0 & 0 & 0 & 0 & 0 & 0\\
0 & 0 & 0 & 0 & 0 & 0 & 0 & 0 & 0 & 0 & \pm 8 & \pm 8 & 0 & 0 & \pm 8 & \pm 8
\end{array}\right),
$$
where signs of the elements are determined by unspecified
bent functions~$g_{\vec{u}}\in\func{B}_2$.
\qed
\end{example}

It is convenient to assume that every one-element subset 
of a vector space is an affine plane of dimension~$0$ 
and that~$\func{B}_0=\func{F}_0$. 
Then under~$r=0$ the above proposition 
gives us all Maiorana--McFarland's bent squares.

In general, Proposition~\ref{Prop.BRECTS.APartConstr} 
allows to construct
$$
(q^m)! c_q(n,m)|\func{B}_{n-m}|^{q^m}
$$
distinct regular bent functions of~$n+m$ variables.
Here $c_q(n,m)$ is the number of distinct 
partitions of~$V_n=\FF_q^n$ into $q^m$ affine planes of dimension~$n-m$. 

To obtain estimates for~$c_q(n,m)$ 
consider some partition $\{E_1,E_2,\ldots,E_{q^m}\}$ counted by $c_q(n,m)$.
Let~$E_i=L_i+\vec{b}_i$, where $L_i<V_n$, $\vec{b}_i\in V_n$,
and
$$
W=L_1\cap L_2\cap\ldots\cap L_{q^m}.
$$
Call the partition~$\{E_i\}$ {\it primitive} if~$W=\{\vec{0}\}$.
Denote by $c_q^*(n,m)$ the number of distinct 
primitive partitions of $V_n$ into planes of dimension~$n-m$.

Suppose that~$\{E_i\}$ is not primitive, that is, $d=\dim W\geq 1$.
Then $V_n$ can be represented as the direct sum
$U\oplus W$, $U<V_n$, $\dim U=n-d$,
and each plane~$E_i$ takes the form
$\{\vec{u}+\vec{w}\colon\vec{u}\in E_i', \vec{w}\in W\}$,
where~$E_i'=E_i\cap U$ is an affine plane of~$U$ of dimension~$n-m-d$.
The planes $E_1',E_2',\ldots,E_{q^m}'$ determine a primitive partition of~$U$.
There are $c_q^*(n-d,m)$ ways to choose such a partition,
$\GaussCoeff{n}{d}{q}$ ways to choose~$W$ and, consequently, 
\begin{equation}\label{Eq.BRECTS.APartCount}
c_q(n,m)=\sum_{d=0}^{n-m}
\GaussCoeff{n}{d}{q} c_q^*(n-d,m).
\end{equation}

Denote by $L_i+L_j$ the subspace of $V_n$ consisting of the 
sums $\vec{v}+\vec{v}'$, 
where $\vec{v}$ runs over $L_i$ and $\vec{v}'$ runs over $L_j$.
For each distinct $i,j\in\{1,2,\ldots,q^m\}$ we have
$$
\dim(L_i\cap L_j)=\dim L_i+\dim L_j-\dim(L_i+L_j)
\geq 2(n-m)-n=n-2m.
$$
If $\dim(L_i\cap L_j)=n-2m$, then $L_i+L_j=V_n$
and $\vec{b}_j-\vec{b}_i=\vec{v}+\vec{v}'$ for some $\vec{v}\in L_i$,
$\vec{v}'\in L_j$. 
It means that
$$
|E_i\cap E_j|=|L_i\cap(L_j+\vec{v}+\vec{v}')|=
|(L_i-\vec{v})\cap(L_j+\vec{v}')|=|L_i\cap L_j|\neq 0,
$$
a contradiction.
Hence
\begin{equation}\label{Eq.BRECTS.APart.LiLj}
\dim(L_i\cap L_j)\geq n-2m+1.
\end{equation}

Under $m=1$ this inequality yields that~$L_1=\ldots=L_q=W$,
where~$\dim W=n-1$. 
Therefore, $c_q^*(1,1)=1$, $c_1^*(n,1)=0$ for $n\geq 2$,
and
$$
c_q(n,1)=\GaussCoeff{n}{n-1}{q}c_q^*(1,1)=\frac{q^n-1}{q-1}.
$$

For the case $q=2$ we also obtain the estimate
$$
c_2(n,2)=\GaussCoeff{n}{n-2}{2}+98\GaussCoeff{n}{n-3}{2}=
\frac{1}{3}(2^n-1)(2^{n-1}-1)(7\cdot 2^{n-1}-13)
$$ 
(cf. Example~\ref{Ex.BRECTS.4Row}) using the following result.

\begin{lemma}\label{Lemma.BRECTS.APartCount}
$c_2^*(2,2)=1$,
$c_2^*(3,2)=98$,
and $c_2^*(n,2)=0$ for $n\geq 4$.
\end{lemma}
\begin{proof}
The first equality is trivial.
Further, if $q=2$ then each partition of $V_3$ into two-element subsets
is also the partition into affine planes. 
There are $\frac{8!}{4!2^4}=105$ such distinct partitions and
$$
c_2^*(3,2)=105-\GaussCoeff{3}{1}{2}c_2^*(2,2)=98
$$ 
of them are primitive.

We prove the third equality by a contradiction.
Suppose that~$n\geq 4$ and $\{E_i=L_i+\vec{b}_i\colon i=1,2,3,4\}$ 
is a primitive partition of~$V_n$ into planes of dimension~$n-2$.
Let us analyze the properties of such hypothetical partition
which help to reveal a contradiction.
\begin{itemize}
\item[A.]
Let $U=L_1+L_2+L_3+L_4$. 
Prove that $U=V_n$. 

Note that $\dim U\geq \dim L_1=n-2$.
If $\dim U=n-2$,
then $L_1=L_2=L_3=L_4$ and the partition $\{E_i\}$ 
is not primitive for $n\geq 3$.

Suppose that $\dim U=n-1$ and let, without loss of generality,
$U=\{(\vec{x},0)\colon\vec{x}\in V_{n-1}\}$.
Denote $U_i=E_i\cap U$.
Since $\{E_i\}$ is a partition, $E_i=U_i+(0,\ldots,0,b_i)$ 
for some $b_i\in\FF_2$, where there are two $0$'s and two $1$'s 
among~$\{b_i\}$.
Assume for simplicity that $b_1=b_2$ and therefore $b_3=b_4$.
Then $U_1\cup U_2=V_{n-1}$ and $U_3\cup U_4=V_{n-1}$.
As we show later, it yields $L_1=L_2$ and $L_3=L_4$.
Now using~\eqref{Eq.BRECTS.APart.LiLj} we obtain
$$
\dim(L_1\cap L_2\cap L_3\cap L_4)=
\dim(L_1\cap L_3)\geq n-3
$$
and the partition $\{E_i\}$ is not primitive for $n\geq 4$.

\item[B.]
Prove that~$d_{ij}=\dim(L_i\cap L_j)=n-3$ for all~$1\leq i<j\leq 4$.

By~\eqref{Eq.BRECTS.APart.LiLj}, the numbers $d_{ij}\in\{n-3,n-2\}$. 
Suppose that $d_{ij}=n-2$ for some $i$ and~$j$, 
say for~$i=3$ and~$j=4$.
Then $L_3=L_4$,
\begin{align*}
\dim(L_1\cap L_2\cap L_3\cap L_4)&=
\dim(L_1\cap L_2\cap L_3)\\
&=\dim(L_1+L_2+L_3)-\dim L_1 -\dim L_2 -\dim L_3\\
&\phantom{~~~~~}+d_{12}+d_{13}+d_{23}\\
&\geq n-3(n-2)+3(n-3) = n-3,
\end{align*}
and the partition $\{E_i\}$ is not primitive for $n\geq 4$.

\item[C.]
Prove that $d_{ijk}=\dim(L_i\cap L_j\cap L_k)\in\{n-4,n-3\}$
for all~$1\leq i<j<k\leq 4$ and moreover $d_{ijk}\neq n-3$ for $n>4$.

Indeed, $d_{ijk}\leq d_{ij}=n-3$ and
\begin{align*}
d_{ijk}&=\dim L_i + \dim(L_j\cap L_k)-\dim(L_i+(L_j\cap L_k))\\
&\geq (n-2)+(n-3)-\dim(L_i+L_j)=n-4.
\end{align*}
If some coefficient $d_{ijk}$ is equal to $n-3$, 
say $d_{123}=n-3$, then
\begin{align*}
\dim((L_1\cap L_2\cap L_3)+L_4)&=
d_{123}+\dim L_4-\dim(L_1\cap L_2\cap L_3\cap L_4)\\
&=(n-3)+(n-2)+0=2n-5.
\end{align*}
On the other hand,
$$
\dim((L_1\cap L_2\cap L_3)+L_4)\leq \dim(L_1+L_4)=n-1
$$
and we obtain the inequality $2n-5\leq n-1$ that doesn't hold for~$n>4$.
\end{itemize}

Gathering A, B and~C, we get
\begin{align*}
\dim(L_1\cap L_2\cap L_3\cap L_4)&=
\sum\dim L_i -
\sum d_{ij}+
\sum d_{ijk}-
\dim\left(\sum L_i\right)\\
&=4(n-2)-6(n-3)+\sum d_{ijk}-n\\
&\geq 4(n-2)-6(n-3)+4(n-4)-n=n-6.
\end{align*}
Consequently, the partition $\{E_i\}$ is primitive only if~$n\leq 6$ and
$\sum d_{ijk}=3n-10$.
Check these conditions.

If $n=4$, then two coefficients $d_{ijk}$ are equal to $n-3=1$,
say $d_{123}=d_{124}=1$. 
Then $\dim(L_1\cap L_2\cap L_3\cap L_4)=1$, a contradiction.
If $n=5$, then there exists~$d_{ijk}=n-3=2$,
a contradiction to the second part of C. 
Finally, if~$n=6$, then
\begin{align*}
\dim(L_1\cap L_3)&=\dim L_1 + \dim L_3 - \dim(L_1 + L_3)\\
&\leq 8-\dim((L_1\cap L_2)+(L_3\cap L_4))\\
&=8-\left(
\dim(L_1\cap L_2)+\dim(L_3\cap L_4)-\dim(L_1\cap L_2\cap L_3\cap L_4)
\right)\\
&=8-(3+3-0)=2
\end{align*}
that contradicts B.
\end{proof}

The above lemma allows to account 
all partitions of~$V_{r+2}$ into planes of dimension~$r$
during the proof of the following result.

\begin{proposition}\label{Prop.BRECTS.APart.2}
Let~$q=2$, $r\geq 2$, $\boxup{f}\in\boxup{\func{B}}_{2,r+2}$ be
constructed by Proposition~\ref{Prop.BRECTS.APartConstr}
and $f(\vec{u},\vec{x})$, $\vec{u}\in V_2$, $\vec{x}\in V_{r+2}$,
be a corresponding bent function.
By separate affine permutations of the variables $\vec{u}$ and $\vec{x}$,
$f$ can be transformed into one of the following functions:
\begin{align}\label{Eq.BRECTS.APart.21}
g(\vec{u},\vec{x})&=
u_1 u_2(g_1(x_3,\vec{y})+g_2(x_3,\vec{y})+g_3(x_3,\vec{y})+g_4(x_3,\vec{y}))\\
\notag
&+u_1(g_1(x_3,\vec{y})+g_3(x_3,\vec{y})+x_1)\\
\notag
&+u_2(g_1(x_3,\vec{y})+g_2(x_3,\vec{y})+x_2)\\
\notag
&+g_1(x_3,\vec{y}),\\
\label{Eq.BRECTS.APart.22}
g(\vec{u},\vec{x})&=
u_1 u_2(g_1(x_3,\vec{y})+g_2(x_3,\vec{y})+g_3(x_2,\vec{y})+
g_4(x_2,\vec{y})+x_2+x_3)\\
\notag
&+u_1(g_1(x_3,\vec{y})+g_3(x_2,\vec{y})+x_1)\\
\notag
&+u_2(g_1(x_3,\vec{y})+g_2(x_3,\vec{y})+x_2)\\
\notag
&+g_1(x_3,\vec{y}),\\
\label{Eq.BRECTS.APart.23}
g(\vec{u},\vec{x})&=
u_1 u_2(g_1(x_3,\vec{y})+g_2(x_1,\vec{y})+g_3(x_1+x_2+x_3,\vec{y})+
g_4(x_2,\vec{y})+x_1)\\
\notag
&+u_1(g_1(x_3,\vec{y})+g_3(x_1+x_2+x_3,\vec{y})+x_2+x_3)\\
\notag
&+u_2(g_1(x_3,\vec{y})+g_2(x_1,\vec{y})+x_2)\\
\notag
&+g_1(x_3,\vec{y}),
\end{align}
where~$\vec{x}=(x_1,x_2,x_3,\vec{y})$, $\vec{y}\in V_{r-1}$, 
$g_i\in\func{B}_r$.
\end{proposition}
\begin{proof}
Consider partitions of~$V_3$ into planes of dimension~$1$.
Denote such a plane~$\{\vec{b},\vec{b}+\vec{e}\}$ by~$[\vec{b},\vec{e}]$.
Let~$\vec{e}_i$ be the vector of $V_3$ having only the $i$th
coordinate nonzero.
It is easy to check that every partition of~$V_3$ 
can be transformed into one of the following
\begin{align*}
&\{[\vec{0},\vec{e}_3], [\vec{e}_2,\vec{e}_3],
[\vec{e}_1,\vec{e}_3], [\vec{e}_1+\vec{e}_2,\vec{e}_3]\},\\
&\{[\vec{0},\vec{e}_3], [\vec{e}_2,\vec{e}_3],
[\vec{e}_1,\vec{e}_2], [\vec{e}_1+\vec{e}_3,\vec{e}_2]\},\\
&\{[\vec{0},\vec{e}_3], [\vec{e}_2,\vec{e}_1],
[\vec{e}_2+\vec{e}_3,\vec{e}_1+\vec{e}_2+\vec{e}_3],
[\vec{e_1}+\vec{e}_3,\vec{e}_2]\}
\end{align*}
by an affine permutation of coordinates.
 
Choose some of these partitions, 
replace each its plane
$[\vec{b},\vec{e}]=\{\vec{b}+\alpha\vec{e}\colon\alpha\in\FF_2\}$ by
$\{(\vec{b},\vec{0})+\alpha(\vec{e},\vec{y})\colon
\alpha\in\FF_2,\ \vec{y}\in V_{r-1}\}$
and obtain the partition of~$V_{r+2}$ 
into the planes of dimension~$r$.
Call such a partition canonical
(for example, a canonical partition was used in 
Example~\ref{Ex.BRECTS.APartConstr}).
Lemma~\ref{Lemma.BRECTS.APartCount} implies that
each partition of~$V_{r+2}$ can be transformed into some canonical 
by an affine permutation of coordinates.
Also introduce a canonical ordering of planes of partitions 
by which the lexicographically minimal vectors of consecutive planes 
are increased (cf. Example~\ref{Ex.BRECTS.APartConstr}).

Let a bent rectangle $\boxup{f}$ be constructed by 
Proposition~\ref{Prop.BRECTS.APartConstr} using a partition~$\{E_\vec{u}\}$.
By an affine permutation of columns of $\boxup{f}$, 
the partition~$\{E_{\vec{u}}\}$ can be transformed into some canonical
and by an affine permutation of rows, 
the canonical ordering of planes can be achieved.
Thus, it is sufficient to consider bent rectangles $\boxup{g}$ 
built by the canonical partitions $\{E_\vec{u}\}$ under the 
canonical ordering and determine the corresponding bent functions
$g(\vec{u},\vec{x})$.
To do this, we can use the stretching equations~\eqref{Eq.BRECTS.StretchFunc}, 
\eqref{Eq.BRECTS.StretchSpectrum} to calculate
the restrictions of~$g$ to~$\vec{x}$ and then utilize the representation
$$
g(u_1,u_2,\vec{x})=\sum_{\alpha_1,\alpha_2\in\FF_2}
(u_1+\alpha_1+1)(u_2+\alpha_2+1)g(\alpha_1,\alpha_2,\vec{x}).
$$

For example, if $\boxup{g}$ is built by the third canonical partition,
then its restrictions to~$\vec{x}$ look as follows:
\begin{align*}
g(0,0,\vec{x})&=
g_1(x_3,\vec{y}),\\
g(0,1,\vec{x})&=
g_2(x_1,\vec{y})+x_2,\\
g(1,0,\vec{x})&=
g_3(x_1+x_2+x_3,\vec{y})+x_2+x_3,\\
g(1,1,\vec{x})&=
g_4(x_2,\vec{y})+x_1+x_3
\end{align*}
and we obtain~\eqref{Eq.BRECTS.APart.23}.
\end{proof}

It is interesting that if $\boxup{f}\in\boxup{\func{B}}_{2,r+2}$ 
is constructed by Proposition~\ref{Prop.BRECTS.APartConstr} and
$f_1(\vec{x})$, 
$f_2(\vec{x})$,
$f_3(\vec{x})$,
$f_4(\vec{x})$
are the restrictions of $f(\vec{u},\vec{x})$ to $\vec{x}$,
then all the functions $f_i+f_j$, $1\leq i < j\leq 4$, are balanced,
that is, they take the values $0$ and $1$ equally often.

Indeed, if $g(\vec{u},\vec{x})$ has one of the 
forms~\eqref{Eq.BRECTS.APart.21}~--- \eqref{Eq.BRECTS.APart.23},
then each sum of two its distinct restrictions to $\vec{x}$ is balanced.
It is easily follows from the form of $g$ and the fact that 
$h(\vec{x})=h^*(x_1,\ldots,x_{k-1},x_{k+1},\ldots,x_{r+2})+x_k$
and all derived functions $\sigma(h)(\vec{x})$, $\sigma\in\AGL_{r+2}$,
are balanced.
By Proposition~\ref{Prop.BRECTS.APart.2}, $f$ can be obtained from~$g$ 
by separate affine permutations of~$\vec{u}$ and $\vec{x}$.
But a permutation of $\vec{u}$ only rearrange the restrictions,
and a permutation of~$\vec{x}$ doesn't change the balance of their sums.

%%%%%%%%%%%%%%%%%%%% References %%%%%%%%%%%%%%%%%%%%%%%%%%%%%%%%%%%%%

\end{sloppypar}
\end{document}